\documentclass[11pt, a4paper]{article}
\usepackage{mathrsfs, mathtools, amsthm, amssymb, thm-restate}  
\usepackage{paralist, tablists}  

\usepackage{graphicx, xcolor, tikz}  
\usetikzlibrary{calc, shapes, decorations.pathreplacing, decorations.markings}  
\usepackage{caption}  
\usepackage[labelformat=simple]{subcaption}  

\usepackage[left=25mm, top=25mm, bottom=25mm, right=25mm]{geometry}  
\usepackage[T1]{fontenc} 

\usepackage[inline]{enumitem}  
\usepackage[square, numbers, sort&compress]{natbib}  
\usepackage[
  bookmarks=true,                
  bookmarksnumbered=true,        
  bookmarksopen=true,            
  plainpages=false,              
  colorlinks=true,               
  linkcolor=blue,                
  citecolor=black,               
  anchorcolor=green,             
  urlcolor=blue,                 
  hyperindex=true                
]{hyperref} 

\newtheorem{theorem}{Theorem}[section] 
\newtheorem{lemma}[theorem]{Lemma}

\newtheorem{case}{Case}
\newtheorem{subcase}{Subcase}[case]

\newtheorem{claim}{Claim}
\theoremstyle{definition}
\newtheorem{definition}{Definition}
\newtheorem{remark}{Remark}

\makeatletter
\def\th@plain{%
  \upshape 
}
\makeatother

\newcommand{\etal}{et~al.\ }
\newcommand{\ie}{i.e.,\ }

\usepackage{marvosym,wasysym}
\usepackage{array}
\usepackage{bm}  

\makeatletter
\renewenvironment{proof}[1][\proofname]{\par
  \pushQED{\qed}%
  \normalfont \topsep6\p@\@plus6\p@\relax
  \trivlist
  \item[\hskip\labelsep
        \bfseries
    #1\@addpunct{.}]\ignorespaces
}{%
  \popQED\endtrivlist\@endpefalse
}
\makeatother

\usepackage[capitalise]{cleveref}  
\crefname{claim}{Claim}{Claims}

\newcommand{\hollownode}[1]{\node[circle, inner sep = 1, fill = white, draw] () at (#1) {}}
\newcommand{\solidnode}[1]{\node[circle, inner sep = 1, fill = black, draw] () at (#1) {}}
\newcommand{\recsolidnode}[1]{\node[rectangle, inner sep = 2, fill = green, draw] () at (#1) {}}

\tikzset{
  on each segment/.style={
    decorate,
    decoration={
      show path construction,
      moveto code={},
      lineto code={
        \path [#1] (\tikzinputsegmentfirst) -- (\tikzinputsegmentlast);
      },
      curveto code={
        \path [#1] (\tikzinputsegmentfirst)
        .. controls (\tikzinputsegmentsupporta) and (\tikzinputsegmentsupportb) ..
        (\tikzinputsegmentlast);
      },
      closepath code={
        \path [#1] (\tikzinputsegmentfirst) -- (\tikzinputsegmentlast);
      },
    },
  },
  mid arrow/.style={postaction={decorate,decoration={
        markings,
        mark=at position .55 with {\arrow[#1]{stealth}}
      }}},
}

\begin{document}

\title{Planar graphs with distance of 3-cycles at least 2 and no cycles of lengths 5, 6, 7}
\author{Tao Wang\footnote{Center for Applied Mathematics, Henan University, Kaifeng, 475004, P. R. China. {\tt Corresponding author: wangtao@henu.edu.cn; https://orcid.org/0000-0001-9732-1617}.} \and Ya-Nan Wang\footnote{School of Mathematics and Statistics, Henan University, Kaifeng, 475004, P. R. China.} \and Xiaojing Yang\footnote{School of Mathematics and Statistics, Henan University, Kaifeng, 475004, P. R. China.}}

\date{June 27, 2024}
\maketitle

\begin{abstract}
Weak degeneracy of a graph is a variation of degeneracy that has a close relationship to many graph coloring parameters. In this article, we prove that planar graphs with distance of $3$-cycles at least 2 and no cycles of lengths $5, 6, 7$ are weakly $2$-degenerate. Furthermore, such graphs can be vertex-partitioned into two subgraphs, one of which has no edges, and the other is a forest.

\end{abstract}

\section{Introduction}
\label{sec1}
A graph $G$ is $k$-\emph{degenerate} if the minimum degree of every subgraph is at most $k$. Equivalently, the vertices of $G$ can be linearly ordered $<$ such that each vertex $v$ has at most $k$ neighbors $u$ with $v < u$. Using Euler's formula, we know that every triangle-free planar graph is $3$-degenerate. Regarding to degeneracy of some planar graphs, see \cite{MR1889505,MR1914478,MR4414782}.

Let $\mathbb{Z}$ be the set of integers, and $\mathbb{Z}^{G}$ be the set of all mappings from $V(G)$ to $\mathbb{Z}$. For a subset $U$ of $V(G)$, we let $f_{-U}: V(G)-U\longrightarrow \mathbb{Z}$ be a mapping defined as $f_{-U}(x)=f(x)-|N_{G}(x)\cap U|$ for all $x \in V(G)-U$. In particular, we write $f_{-v}$ for $f_{-\{v\}}$. A graph $G$ is \emph{strictly $f$-degenerate} if every subgraph $R$ of $G$ contains a vertex $v$ with $d_{R}(v) < f(v)$. According to definition, $G$ is strictly $1$-degenerate if and only if it has no edges, which is an independent set; $G$ is strictly $2$-degenerate if and only if it is a forest.

Motivated by the study of greedy algorithms for graph coloring, Bernshteyn and Lee \cite{MR4606413} introduced a weaker version of ordinary degenerate, named weakly degenerate. 

\begin{definition}
Given a graph $G$ and $f \in \mathbb{Z}^{G}$. For a vertex $u \in V(G)$, the operation $\textsf{Delete}_{u}(G, f)$ outputs the graph $G_{-u} = G - u$ and the function $f_{-u}$. 
If $f(u) \geq 0$ and $f_{-u} \geq 0$, we say the operation $\textsf{Delete}_{u}$ is \emph{legal}. 
\end{definition}

\begin{definition}
Given a graph $G$ and $f \in\mathbb{Z}^{G}$. For two adjacent vertices $a$ and $b$ in $G$, the operation $\textsf{DeleteSave}_{(a, b)}$ on $(G, f)$ outputs the graph $G_{-a} = G - a$ and the function $f'$ defined as
\begin{align*}
f'(v) \coloneqq
&\begin{cases}
f(v) - 1, & \text{if $v$ is a neighbor of $a$ and $v \neq b$};\\[0.3cm]
f(v), & \text{otherwise}.
\end{cases}
\end{align*}
If $f(a) > f(b)$ and $f' \geq 0$, we say the operation $\textsf{DeleteSave}_{(a, b)}$ is \emph{legal} on $(G, f)$. 
\end{definition}

A graph $G$ is \emph{weakly $f$-degenerate} if it can be reduced to a null graph through a sequence of legally operations \textsf{Delete} and \textsf{DeleteSave}. For a positive integer $k$, we define $G$ as \emph{weakly $k$-degenerate} if it is weakly $f$-degenerate, where $f$ is a constant function with a value of $k$. The \emph{degeneracy} $\textsf{d}(G)$ of $G$ is the minimum integer $k$ such that $G$ is $k$-degenerate. The \emph{weak degeneracy} $\textsf{wd}(G)$ of $G$ is the minimum integer $k$ for which $G$ is weakly $k$-degenerate. Bernshteyn and Lee \cite{MR4606413} established a chain of inequalities:
\begin{equation}\label{eq1_v1}
\chi(G) \leq \chi_{\ell}(G) \leq \chi_{\textsf{DP}}(G) \leq \textsf{wd}(G) + 1 \leq \textsf{d}(G) + 1, 
\end{equation}
where $\chi_{\textsf{DP}}(G)$ is the DP-chromatic number of $G$.

According to \cref{eq1_v1}, various graph coloring parameters are upper bounded by $\textsf{wd}(G) + 1$. Consequently, there is significant interest in finding various sufficient conditions for planar graphs to be weakly $k$-degenerate. Wang \cite{MR4564473} provided two such conditions for a plane graph to be weakly $2$-degenerate. A best possible condition for a toroidal graph to be weakly $3$-degenerate was given in \cite{MR4746933}.

Han \etal \cite{MR4663366} proved that planar graphs of girth 5 are weakly $2$-degenerate. 
\begin{theorem}[Han et al. \cite{MR4663366}]
If $G$ is a triangle-free planar graph in which no 4-cycle is normally adjacent to a cycle of length at most five, then $\textsf{wd}(G) \leq 2$. 
\end{theorem}

Liu \etal \cite{MR3886261} proved that planar graphs with triangles at distance greater than two and no $5$-, $6$-, $7$-cycles are DP-$3$-colorable. This paper strengthens this result. Let $\mathrm{dist}^{\Delta}$ be the smallest distance between triangles, \ie $\mathrm{dist}^{\Delta} = \min\{d_{G}(v_{1}, v_{2}): \text{$v_{i} \in V(T_{i})$ and $T_{i}$'s are triangles of $G$}\}$.

\begin{theorem}\label{Weak}
Every planar graph with $\mathrm{dist}^{\Delta} \geq 2$ and without $5$-, $6$-, $7$-cycles is weakly $2$-degenerate. 
\end{theorem}

We first prove the following structural result, which can imply \cref{Weak}. 
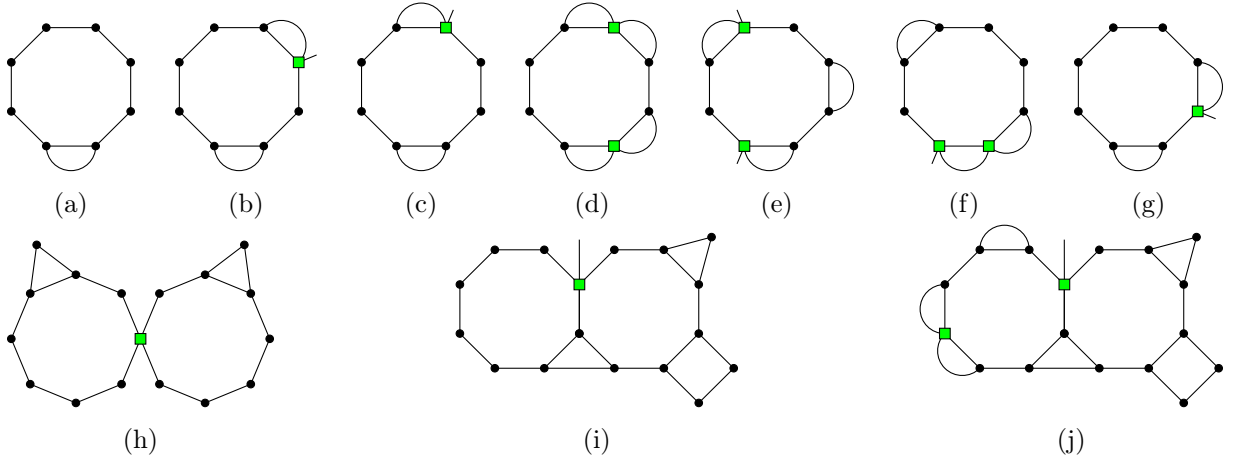
\begin{figure}
\def\s{1}
\def\8cycle{
\foreach \i in {1, 2, ..., 8}
{
\def\pointname{v\i}
\coordinate (\pointname) at ($(\i*45+22.5:\s)$);
}
\draw (v1)--(v2)--(v3)--(v4)--(v5)--(v6)--(v7)--(v8)--cycle;
\foreach \i in {1, 2, ..., 8}{\solidnode{v\i};}
}
\subcaptionbox{\label{Ca}}{\begin{tikzpicture}[scale=0.85]
\8cycle
\draw (v5) arc (180:360:0.383);
\end{tikzpicture}}
\subcaptionbox{\label{Cb}}{\begin{tikzpicture}[scale=0.85]
\8cycle
\draw (v5) arc (180:360:0.383);
\draw (v8) arc (-45:135:0.383);
\draw (v8) -- ($1.3*(v8)$);
\recsolidnode{v8};
\end{tikzpicture}}
\subcaptionbox{\label{Cc}}{\begin{tikzpicture}[scale=0.85]
\8cycle
\draw (v1) arc (0:180:0.383);
\draw (v5) arc (180:360:0.383);
\draw (v1) -- ($1.3*(v1)$);
\recsolidnode{v1};
\end{tikzpicture}}
\subcaptionbox{\label{Cd}}{\begin{tikzpicture}[scale=0.85]
\8cycle
\draw (v8) arc (-45:135:0.383);
\draw (v1) arc (0:180:0.383);
\draw (v5) arc (180:360:0.383);
\draw (v6) arc (-135:45:0.383);
\recsolidnode{v1};
\recsolidnode{v6};
\end{tikzpicture}}
\subcaptionbox{\label{Ce}}{\begin{tikzpicture}[scale=0.85]
\8cycle
\draw (v2) -- ($1.3*(v2)$);
\draw (v2) arc (45:225:0.383);
\draw (v5) arc (180:360:0.383);
\draw (v5) -- ($1.3*(v5)$);
\draw (v7) arc (-90:90:0.383);
\recsolidnode{v2};
\recsolidnode{v5};
\end{tikzpicture}}
\subcaptionbox{\label{Cf}}{\begin{tikzpicture}[scale=0.85]
\8cycle
\draw (v2) arc (45:225:0.383);
\draw (v5) -- ($1.3*(v5)$);
\draw (v5) arc (180:360:0.383);
\draw (v6) arc (-135:45:0.383);
\recsolidnode{v5};
\recsolidnode{v6};
\end{tikzpicture}}
\subcaptionbox{\label{Cg}}{\begin{tikzpicture}[scale=0.85]
\8cycle
\draw (v5) arc (180:360:0.383);
\draw (v7) arc (-90:90:0.383);
\draw (v7) -- ($1.3*(v7)$);
\recsolidnode{v7};
\end{tikzpicture}}
\subcaptionbox{\label{C-special}}{\begin{tikzpicture}[scale=0.85]
\begin{scope}[shift={(-1,0)},rotate=22.5]
\8cycle
\coordinate (H) at ($(v2)!1!60:(v1)$);
\draw (v1) -- (H) -- (v2);
\solidnode{H};
\end{scope}
\begin{scope}[shift={(1,0)},rotate=22.5]
\8cycle
\coordinate (H) at ($(v1)!1!60:(v8)$);
\draw (v8) -- (H) -- (v1);
\solidnode{H};
\end{scope}
\recsolidnode{v3};
\end{tikzpicture}}\hfill
\subcaptionbox{\label{C-f3f4a}}{\begin{tikzpicture}[scale=0.85]
\begin{scope}[shift={(-0.925,0)}]
\foreach \i in {1, 2, ..., 8}
{
\def\pointname{u\i}
\coordinate (\pointname) at ($(\i*45+22.5:\s)$);
}
\draw (u1)--(u2)--(u3)--(u4)--(u5)--(u6)--(u7)--(u8)--cycle;
\foreach \i in {1, 2, ..., 8}{\solidnode{u\i};}
\coordinate (H) at ($(u2)!1!60:(u1)$);
\end{scope}
\begin{scope}[shift={(0.925,0)}]
\8cycle
\coordinate (H) at ($(v1)!1!60:(v8)$);
\draw (v8) -- (H) -- (v1);
\solidnode{H};
\coordinate (H1) at ($(v7)!1!90:(v6)$);
\coordinate (H2) at ($(v6)!1!-90:(v7)$);
\draw (v6) -- (H2) -- (H1) -- (v7);
\solidnode{H1};
\solidnode{H2};
\end{scope}
\draw (u6) -- (v5);
\draw (v3) -- ($(u1)-(u8)+(v2)-(v3)$);
\recsolidnode{v3};
\end{tikzpicture}}\hfill
\subcaptionbox{\label{C-f3f4b}}{\begin{tikzpicture}[scale=0.85]
\begin{scope}[shift={(-0.925,0)}]
\foreach \i in {1, 2, ..., 8}
{
\def\pointname{u\i}
\coordinate (\pointname) at ($(\i*45+22.5:\s)$);
}
\draw (u1)--(u2)--(u3)--(u4)--(u5)--(u6)--(u7)--(u8)--cycle;
\foreach \i in {1, 2, ..., 8}{\solidnode{u\i};}
\coordinate (H) at ($(u2)!1!60:(u1)$);
\end{scope}
\begin{scope}[shift={(0.925,0)}]
\8cycle
\coordinate (H) at ($(v1)!1!60:(v8)$);
\draw (v8) -- (H) -- (v1);
\solidnode{H};
\coordinate (H1) at ($(v7)!1!90:(v6)$);
\coordinate (H2) at ($(v6)!1!-90:(v7)$);
\draw (v6) -- (H2) -- (H1) -- (v7);
\solidnode{H1};
\solidnode{H2};
\end{scope}
\draw (u6) -- (v5);
\draw (v3) -- ($(u1)-(u8)+(v2)-(v3)$);
\draw (u1) arc (0:180:0.383);
\draw (u3) arc (90:270:0.383);
\draw (u4) arc (135:315:0.383);
\recsolidnode{v3};
\recsolidnode{u4};
\end{tikzpicture}}
\caption{A set of unavoidable configurations, where a semicircle represents a path of length two or three with interior vertices having degree three in $G$, a solid point represents a vertex of degree three, and a green square point represents a vertex of degree four in $G$.}
\label{UC}
\end{figure}

\begin{theorem}\label{STR}
Let $G$ be a plane graph with $\mathrm{dist}^{\Delta} \geq 2$ and without cycles of length $5, 6, 7$. Then one of the following two conclusions holds: 
\begin{enumerate}
\item there is a vertex of degree at most two. 
\item there exists a degree-restricted subgraph isomorphic to a configuration as depicted in \cref{UC}.
\end{enumerate}
\end{theorem}

Actually, \cref{STR} can also be used to prove a result on a special strictly $f$-degenerate transversal, which can imply an $(\mathcal{I}, \mathcal{F})$-coloring of $G$. 

A graph $G$ is \emph{$(\mathcal{I}, \mathcal{F})$-partitionable} if its vertices set can be divided into two parts, one part is a forest and the other has no edges. Borodin and Glebov \cite{MR1918259} confirmed that every planar graph with girth at least $5$ is \emph{$(\mathcal{I}, \mathcal{F})$}-partitionable. Liu \etal \cite{MR3886261} proved that planar graphs without $4$-, $6$- and $8$-cycles are $(\mathcal{I}, \mathcal{F})$-partitionable, and Kang \etal \cite{arXiv:2303.04648} proved that planar graphs without $4$-, $6$- and $9$-cycles are $(\mathcal{I}, \mathcal{F})$-partitionable. In this paper, we are interested in the $(\mathcal{I}, \mathcal{F})$-partition of planar graphs without $5$-, $6$- and $7$-cycles. 

\begin{theorem}\label{IF:partition}
Every planar graph with $\mathrm{dist}^{\Delta} \geq 2$ and without cycles of length $5, 6, 7$ is $(\mathcal{I}, \mathcal{F})$-partitionable.
\end{theorem}

To end this section, we give some notations. An $l$-vertex, or an $l^{+}$-vertex, or an $l^{-}$-vertex means a vertex with degree $l$, or at least $l$, or at most $l$. The notions of an $l$-face, an $l^{+}$-face and an $l^{-}$-face are defined similarly. For a face $f\in F$, if the vertices on $f$ are cyclically listed as $u_{1}, u_{2}, u_{3}, \dots, u_{k}$, then we say that $f$ is a $(d(u_{1}), d(u_{2}),\dots, d(u_{k}))$-face. A face is \emph{light} if the vertices on this face are all $3$-vertices. If $ab$ is a common edge of a $4^{-}$-face and a $7^{+}$-face, then we say $ab$ \emph{controls} the smaller $4^{-}$-face. For a given face $g$, we say a facial $(d_{1}, d_{2},\dots, d_{t})$-walk on $g$ is a \emph{$d$- or $d^{-}$-controlling walk} if each edge on the walk controls a face other than $g$ with size $d$ or at most $d$, respectively. If all vertices on the walk are distinct, we say the walk is a \emph{$d$- or $d^{-}$-controlling path}. Assume $abc$ is a facial path on a $7^{+}$-face $f$, we say $b$ is \emph{rich} with respect to $f$ if neither $ab$ nor $bc$ controls a $4^{-}$-face, \emph{semi-rich} with respect to $f$ if exactly one of $ab$ and $bc$ controls a $4^{-}$-face, and \emph{poor} with respect to $f$ if each of $ab$ and $bc$ controls a $4^{-}$-face. We say a facial path is \emph{maximal} if every edge on $f$ controls a $4^{-}$-face, but the edges on the path immediately after and before it control $5^{+}$-faces. Let $h$ be a face with $d(h) \geq 7$, and let $s_{0}(h)$ denote the number of vertices on $h$ that are not on any $4^{-}$-controlling path.

\section{Proof of \cref{STR}}
\label{sec:2}
Assume that $G$ is a counterexample to \cref{STR}. Then the minimum degree of $G$ is at least three, and there exists no degree-restricted subgraph isomorphic to a configuration as depicted in \cref{UC}.

\begin{lemma}\label{str}
Assuming $G$ is a plane graph with distance of triangles at least 2 and no cycles of lengths 5, 6, 7. Then the following structural results hold:
\begin{enumerate}[label = (\arabic*)]
\item\label{Lem1:i} A $3$-cycle and a $4^{-}$-cycle have no edges in common. 
\item\label{Lem1:iii} There are no $6$-faces. 
\item A $7$-face adjacent to a $4^{-}$-face should be as depicted in \cref{4-7-face}.
\item\label{Lem1:v} A $3$-face is not adjacent to any $7^{-}$-face.
\item\label{Lem1:vi} A $4$-face is not adjacent to any $6^{-}$-face. 
\item If an $8$-face has a cut vertex incident to two $4^{-}$-faces, then it should be as depicted in \cref{4-8-face}.
\end{enumerate}
\end{lemma}

\begin{figure}
\centering
\begin{tikzpicture}
\def\s{1.5}
\coordinate (O) at (0, 0);
\coordinate (NE) at (0.5*\s, 0.5*\s);
\coordinate (W) at (-\s, 0);
\coordinate (S) at (0, -\s);
\coordinate (SW) at (-\s, -\s);
\coordinate (y) at (210:0.7*\s);
\coordinate (z) at (240:0.7*\s);
\fill[gray] (O)--(W)--(SW)--(S)--(O)--(z)--(y)--cycle;
\fill[brown] (S)--(NE)--(W)--(O);
\draw (S)--(NE)--(W);
\draw (O)--(y)--(z)--(O);
\draw (W)--(SW)--(S)--(O)--cycle;
\hollownode{O};
\hollownode{NE};
\hollownode{W};
\hollownode{S};
\hollownode{SW};
\hollownode{y};
\hollownode{z};
\end{tikzpicture}
\caption{A $7$-face adjacent to a $4$-face.}
\label{4-7-face}
\end{figure}
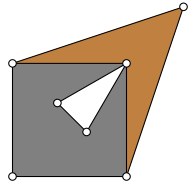

\begin{proof}
Suppose that $f$ is a $6$-face. Since there are no $6$-cycles in $G$, the boundary of $f$ is not a $6$-cycle. Then the boundary of $f$ consists of two triangles, this contradicts the hypothesis that triangles have distance at least two from each other. Hence, there are no $6$-faces in $G$. 

Let $g$ be a $7$-face. Similarly, the boundary of $g$ is not a $7$-cycle, so it can only be a triangle and a $4$-cycle. Let $[xyzxabc]$ be a $7$-face. It follows from \cref{str}\ref{Lem1:i} that none of $xy, yz$ and $xz$ can be incident with a $4^{-}$-face. Similarly, none of $xa, ab, bc$ and $cx$ can be incident with a $3$-face. If the $7$-face $g$ is adjacent to a $4$-face $h$, then $h$ have two consecutive common edges with $xabc$. Hence, it must be as drawn in \cref{4-7-face}. 
\end{proof}

Assume $f$ is an $8$-face that is bounded by an $8$-cycle. We say that $f$ is \emph{special} if it is incident with seven $3$-vertices and a semi-rich $4$-vertex, adjacent to two $3$-faces and two $4$-faces, precisely three of them are light $4^{-}$-faces, and the other $4^{-}$-face containing at least two $4^{+}$-vertices. See an illustration in \cref{special}.

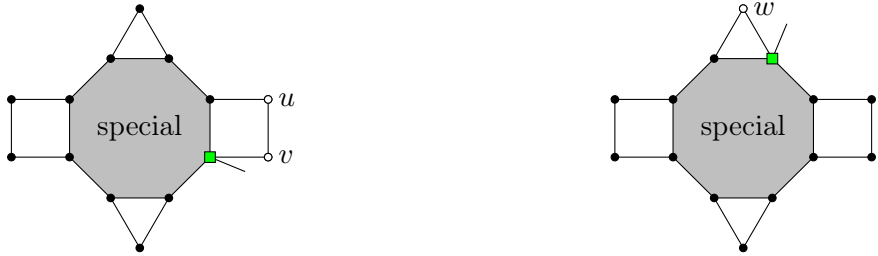
\begin{figure}
\def\s{1}
\def\special{
\foreach \i in {1, 2, ..., 8}
{
\def\pointname{v\i}
\coordinate (\pointname) at ($(\i*45+22.5:\s)$);
}
\draw[fill = lightgray] (v1)--(v2)--(v3)--(v4)--(v5)--(v6)--(v7)--(v8)--cycle;
\foreach \i in {1, 2, ..., 8}{\solidnode{v\i};}
\coordinate (N) at ($(v2)!1!60:(v1)$);
\coordinate (S) at ($(v6)!1!60:(v5)$);
\coordinate (W1) at ($(v3)!1!-90:(v4)$);
\coordinate (W2) at ($(v4)!1!90:(v3)$);
\coordinate (E1) at ($(v8)!1!90:(v7)$);
\coordinate (E2) at ($(v7)!1!-90:(v8)$);
\draw (v1) -- (N) -- (v2);
\draw (v3) -- (W1) -- (W2) -- (v4);
\draw (v5) -- (S) -- (v6);
\draw (v7) -- (E2) -- (E1) -- (v8);
\foreach \v in {N,W1,W2,S,E1,E2}
{\solidnode{\v};}
\node at (0, 0) {special}; 
}
\centering
\subcaptionbox{Type I, where at least one of $u$ and $v$ is a $4^{+}$-vertex.\label{special-1}}[0.55\linewidth]{\begin{tikzpicture}
\special;
\draw (v7) -- ($1.5*(v7)$);
\recsolidnode{v7};
\hollownode{E1};
\hollownode{E2};
\node[right] at (E1) {$u$};
\node[right] at (E2) {$v$};
\end{tikzpicture}}
\subcaptionbox{Type II, where $w$ is a $4^{+}$-vertex.\label{special-2}}[0.4\linewidth]{\begin{tikzpicture}
\special;
\draw (v1) -- ($1.5*(v1)$);
\recsolidnode{v1};
\hollownode{N};
\node[right] at (N) {$w$};
\end{tikzpicture}}
\caption{Two types of special $8$-faces.}
\label{special}
\end{figure}

\begin{lemma}\label{lem:N-special8faces}
A $4$-vertex cannot be incident to two special $8$-faces at opposite.
\end{lemma}
\begin{proof}
Let $f_{1}, f_{2}, f_{3}$ and $f_{4}$ be four incident faces of a $4$-vertex $v$. Assume $f_{2}$ and $f_{4}$ are two special $8$-faces. By the definition of special $8$-faces, $v$ is a semi-rich $4$-vertex of the incident special $8$-faces $f_{2}$ and $f_{4}$. It follows that exactly one of $f_{1}$ and $f_{3}$, say $f_{1}$, is a $4^{-}$-face. Furthermore, $f_{1}$ is a $4$-face incident with precisely two $4^{+}$-vertices. Hence, $G$ contains a copy of \cref{C-special}.
\end{proof}

We use the famous discharging method to finish the proof. Let $\mu(v) = 2d(v) - 6$ be the initial charge for each vertex $v \in V(G)$ and $\mu(f) = d(f) - 6$ be the initial charge of each face $f \in F(G)$. By the Handshaking theorem and Euler's formula, we have that 
\[
\sum_{x \in V(G) \cup F(G)} \mu(x) = \sum_{v \in V(G)}(2d(v) - 6) + \sum_{f \in F(G)}(d(f) - 6) = -12.
\]
Applying the following discharging rules, we obtain a new charge function $\mu^{*}$. In this procedure, the total sum of charges is preserved. However, we show that $\mu^{*} \geq 0$ for each vertex and each face, a contradiction. For convenience, we use $\tau(a \rightarrow b)$ to denote the charges transferred from the element $a$ to the element $b$. 

The discharge rules are as follows: 
\begin{enumerate}[label = \textbf{R-\arabic*.}, ref = R-\arabic*]
\item\label{R1} Each $4^{-}$-face gets $1$ from each $4^{+}$-vertex on it.

\item\label{R2} Let $f$ and $g$ be two adjacent faces. Then 
\begin{align*}
\tau(g \rightarrow f) = 
&\begin{cases}
          1, & \text{if $f$ is a $(3, 3, 3)$-face.}\\[0.2cm]
\frac{2}{3}, & \text{if $f$ is a $(3, 3, 4^{+})$-face.}\\[0.2cm]
\frac{1}{3}, & \text{if $f$ is a $(3, 4^{+}, 4^{+})$-face.}\\[0.2cm]
\frac{1}{2}, & \text{if $f$ is a $(3, 3, 3, 3)$-face.}\\[0.2cm]
\frac{1}{4}, & \text{if $f$ is a $(3, 3, 3, 4^{+})$-face.}\\[0.2cm]
\frac{1}{2}, & \text{if $g$ is a $(3^{+}, 4^{+}, 4^{+}, 4^{+})$-face and it has a common $(4^{+}, 4^{+})$-edge with $f$.}
\end{cases}
\end{align*}

\item\label{R3} Assume $v$ is a $4$-vertex incident with four faces $f_{1}, f_{2}, f_{3}$ and $f_{4}$ in a cyclic order, and $v$ is incident to at most one $4^{-}$-face. 

\begin{enumerate}
\item\label{R3a} If $f_{1}$ is a $4^{-}$-face, $f_{2}$ is a special $8$-face, and $f_{4}$ is not a special $8$-face, then $v$ sends $\frac{1}{2}$ to $f_{2}$, and $\frac{1}{4}$ to each of $f_{3}$ and $f_{4}$.

\item\label{R3b} If $f_{1}$ is a $4^{-}$-face, neither $f_{2}$ nor $f_{4}$ is a special $8$-face, then $v$ sends $\frac{1}{2}$ to $f_{3}$, and $\frac{1}{4}$ to each of $f_{2}$ and $f_{4}$.

\item\label{R3c} If $v$ is not incident to any $4^{-}$-face, then $v$ sends $\frac{1}{2}$ to each incident face.

\end{enumerate}

\item\label{R4} Let $v$ be a $5$-vertex on an $8$-face $f$. Then 
\begin{align*}
\tau(v \rightarrow f) = 
&\begin{cases}
\frac{3}{4}, & \text{if $v$ is semi-rich with respect to $f$.}\\[0.2cm]
\frac{1}{2}, & \text{otherwise.}
\end{cases}
\end{align*}

\item\label{R5} Each $8$-face receives $1$ from each incident $6^{+}$-vertex.
\end{enumerate}

\begin{lemma}
For every $z \in V(G) \cup F(G)$, we have $\mu^{*}(z) \geq 0$.
\end{lemma}

\begin{proof}
Let $v$ be an arbitrary vertex in $G$. If $d(v)=3$, then $\mu^{*}(v)=\mu(v)=0$. Let $v$ be a $4$-vertex incident with four faces $f_{1}, f_{2}, f_{3}$ and $f_{4}$ in a cyclic order. We have the following scenarios:

$\bullet$ Suppose that $v$ is incident with at least two $4^{-}$-faces. Since there are no adjacent $4^{-}$-faces, $v$ is incident with precisely two $4^{-}$-faces at opposite. By \ref{R1}, $v$ gives 1 to each incident $4^{-}$-face, then $\mu^{*}(v) = 2 \times 4 - 6 - 1\times 2 = 0$.

$\bullet$ If $f_{1}$ is a $4^{-}$-face, then by \cref{lem:N-special8faces}, it is impossible that both $f_{2}$ and $f_{4}$ are special $8$-faces. 

$\bullet$ Suppose that $f_{1}$ is a $4^{-}$-face, $f_{2}$ is a special $8$-face, but $f_{4}$ is not. By \ref{R1} and \ref{R3a}, $v$ sends $1$ to $f_{1}$, $\frac{1}{2}$ to $f_{2}$, and $\frac{1}{4}$ to each of $f_{3}$ and $f_{4}$. Then $\mu^{*}(v) = 2 \times 4 - 6 - 1 -\frac{1}{2} - \frac{1}{4}\times 2 = 0$.

$\bullet$ Suppose that $f_{1}$ is a $4^{-}$-face, and neither $f_{2}$ nor $f_{4}$ is a special $8$-face. By \ref{R1} and \ref{R3b}, $v$ sends $1$ to $f_{1}$, $\frac{1}{2}$ to $f_{3}$, and $\frac{1}{4}$ to each of $f_{2}$ and $f_{4}$. Then $\mu^{*}(v) = \mu(v) - 1 - \frac{1}{2} - \frac{1}{4} \times 2 = 0$. 

$\bullet$ If every face incident to $v$ is a $5^{+}$-face, then it should send $\frac{1}{2}$ to each incident face by \ref{R3c}, thus $\mu^{*}(v) = 2 \times 4 - 6 - \frac{1}{2} \times 4 = 0$.

Let $d(v)=5$. Since $4^{-}$-faces are not adjacent, $v$ is incident to at most two $4^{-}$-faces. According to \ref{R4}, when $v$ is semi-rich with respect to an $8$-face, it gives $\frac{3}{4}$ to each such 8-face; if $v$ is a rich or poor $5$-vertex on an $8$-face, then it gives $\frac{1}{2}$ to each such face. If $v$ is incident with two $4^{-}$-faces, then $\mu^{*}(v)\geq 2 \times 5 - 6 - 1\times 2-\frac{3}{4}\times 2-\frac{1}{2}=0$. If $v$ is incident with precisely one $4^{-}$-face, then $\mu^{*}(v)\geq 2 \times 5 - 6 - 1 -\frac{3}{4}\times 2-\frac{1}{2}\times 2 > 0$. If every face incident with $v$ is a $5^{+}$-face, then $\mu^{*}(v)\geq 2 \times 5 - 6 - \frac{1}{2}\times 5 > 0$.

Let $d(v)\geq 6$. According to the discharging rules, $v$ sends at most $1$ to each incident face, implying $\mu^{*}(v) \geq 2d(v) - 6 - 1 \times d(v) \geq 0$.

Next, we check $\mu^{*}$ for every face. Let $f$ be an arbitrary $4^{-}$-face. By \cref{str}\ref{Lem1:v} and \cref{str}\ref{Lem1:vi}, all faces adjacent to $f$ are $7^{+}$-faces. If there is no $4^{+}$-vertices on $f$, then it receives $1$ from each adjacent face, implying $\mu^{*}(f) = 3 - 6 + 1 \times 3 = 0$. If $f$ is a $(3, 3, 4^{+})$-face, then it receives $\frac{2}{3}$ from each adjacent face, and $1$ from the incident $4^{+}$-vertex, thus $\mu^{*}(f) = 3 - 6 + 1 + \frac{2}{3} \times 3 = 0$. If there are exactly two $4^{+}$-vertices, then $\mu^{*}(f) = 3 - 6 + 1 \times 2 + \frac{1}{3} \times 3 = 0$ by \ref{R2}. If $f$ is a $(4^{+}, 4^{+}, 4^{+})$-face, then $\mu^{*}(f) = 3 - 6 + 1 \times 3 = 0$ by \ref{R2}. If $f$ is a $(3, 3, 3, 3)$-face, then it receives $\frac{1}{2}$ from each adjacent face, thus $\mu^{*}(f) = 4 - 6 + \frac{1}{2} \times 4 = 0$. If $f$ is a $(3, 3, 3, 4^{+})$-face, then it receives $1$ from the incident $4^{+}$-vertex, and $\frac{1}{4}$ from each adjacent face, thus $\mu^{*}(f) = 4 - 6 + \frac{1}{4} \times 4 + 1 = 0$. If $f$ is incident with precisely two $3$-vertices, then it receives $1$ from each of the two incident $4^{+}$-vertices, thus $\mu^{*}(f) = 4 - 6 + 1 \times 2 = 0$. If $f$ is a $(3, 4^{+}, 4^{+}, 4^{+})$-face, then it sends $\frac{1}{2}$ to each adjacent face that has a common $(4^{+}, 4^{+})$-edge with $f$, and thus $\mu^{*}(f) = 4 - 6 + 1 \times 3 - \frac{1}{2} \times 2 = 0$. If $f$ is a $(4^{+}, 4^{+}, 4^{+}, 4^{+})$-face, then it sends $\frac{1}{2}$ to each adjacent face, and thus $\mu^{*}(f) = 4 - 6 + 1 \times 4 - \frac{1}{2} \times 4 = 0$. 

Since there are no $5$-cycles, there are no $5$-faces. By \cref{str}\ref{Lem1:iii}, there are no $6$-faces in $G$. For a $7$-face $f$, if it is adjacent to a $4^{-}$-face, then it must be as depicted in \cref{4-7-face}, thus it sends at most $\frac{1}{4} \times 2$ via the incident edges by \ref{R2}, thus $\mu^{*}(f) \geq 7 - 6 - \frac{1}{4} \times 2 = \frac{1}{2} > 0$. 

Let $f$ be an $8^{+}$-face. 

\begin{claim}\label{Claim1}
If $P$ is a $4$-controlling path on $f$, then $f$ sends at most $\frac{1}{2}$ in total via the edges on $P$.
\end{claim}
\begin{proof}
Let $P=u_{1}u_{2}\dots u_{t}$. If $P$ has length one, then $f$ sends at most $\frac{1}{2}$ via the edge of $u_{1}u_{2}$. Assume the length of $P$ is at least two. Since there are no adjacent $4^{-}$-faces, the interior vertices $u_{2},\dots,u_{t-1}$ are all $4^{+}$ vertices, thus $f$ sends at most $\frac{1}{4}$ via each of $u_{1}u_{2}$ and $u_{t-1}u_{t}$ and no charge via the edge $u_{2}u_{3}, u_{3}u_{4},\dots, u_{t-2}u_{t-1}$, thus it sends at most $ \frac{1}{4}\times 2 = \frac{1}{2}$ via the path $P$.
\end{proof}

\begin{claim}\label{Claim2}
Let $P = xyzu$ be a $4^{-}$-controlling path on $f$. Then $f$ sends at most $\frac{11}{12}$ in total via the edges on $P$. 
\end{claim}
\begin{proof}
Since there are no adjacent $4^{-}$-faces, we have $d(y), d(z) \geq 4$. If the three controlled faces are $4$-faces, then $f$ sends at most $\frac{1}{4}$ via each of $xy$ and $zu$, and no charge via $yz$, thus it sends at most $\frac{1}{4} \times 2 = \frac{1}{2} < \frac{11}{12}$ via the path $P$. Since the distance of triangles is least two, $P$ controls at most one $3$-face. If $xy$ or $zu$, say $xy$, controls a $3$-face, then $f$ sends at most $\frac{2}{3}$ via $xy$, $\frac{1}{4}$ via $zu$, and no charge via $yz$, thus it sends at most $\frac{2}{3} + \frac{1}{4} = \frac{11}{12}$ via the path $P$. If $yz$ controls a $3$-face, then $f$ sends at most $\frac{1}{3}$ via $yz$, at most $\frac{1}{4}$ via each of $xy$ and $zu$, thus it sends at most $\frac{1}{3} + \frac{1}{4}\times 2 < \frac{11}{12}$ via the path $P$. 
\end{proof}

Similarly, we have the following claims for the $4^{-}$-controlling path of length at most two.
\begin{claim}\label{Claim3}
Let $P$ be a $4^{-}$-controlling path of length two on $f$. Then $f$ sends at most $\frac{11}{12}$ in total via the edges on $P$. 
\end{claim}

\begin{claim}\label{Claim4}
If $uv$ is a $4^{-}$-controlling path on $f$, then $f$ sends at most $1$ via $uv$. 
\end{claim}

\begin{remark}\label{Remark}
In the proof of the above four claims, when we give an upper bound on the total charge that $v$ sends via a $4^{-}$-controlling path $P$, we treat all vertices on the controlled $4^{-}$-faces as $3$-vertices except the interior vertices of the path $P$. Note that all interior vertices on $P$ are $4^{+}$-vertices. Hence, when we know another $4^{+}$-vertex is on the controlled face, it will ``save'' at least $\frac{1}{4}$. For example, if $P$ is a maximal $4$-controlling path of length one, then $f$ sends at most $\frac{1}{2}$ via the path $P$; moreover, if the controlled face contains precisely one $4^{+}$-vertex, then it ``saves'' $\frac{1}{4}$, that is, $f$ sends $\frac{1}{2} - \frac{1}{4} = \frac{1}{4}$ via the path $P$, this is consistent with \ref{R2}. If the controlled face is incident with precisely two $4^{+}$-vertices, then it ``saves'' $\frac{1}{4} \times 2 = \frac{1}{2}$, so $f$ sends no charge via the path $P$, this is also consistent with \ref{R2}.
\end{remark}

Let $t'_{3} = \textrm{the number of 3-faces that are adjacent to $f$}$. Thus there are at most $\left\lfloor \frac{d(f) - s_{0}(f) - 2t'_{3}}{2}\right\rfloor$ $4$-controlling paths that are disjoint from controlled $3$-faces. By \cref{Claim1,Claim2,Claim3,Claim4}, 
\begin{equation}\label{eq1}
\mu^{*}(f) \geq \mu(f) - 1 \times t'_{3} - \frac{1}{2} \left\lfloor \frac{d(f) - s_{0}(f) - 2t'_{3}}{2}\right\rfloor = \left(d(f) - \frac{1}{2} \left\lfloor \frac{d(f) - s_{0}(f)}{2}\right\rfloor\right) - 6 - \frac{1}{2} t'_{3}. \tag{$\star$}
\end{equation}

Since the distance of triangles is at least two, we have $t'_{3} \leq\lfloor \frac{d(f)}{3}\rfloor$. It is easy to verify that $\mu^{*}(f)\geq 0$ when $f$ is a $10^{+}$-face or a $9$-face with $t'_{3} \leq 2$. If $d(f) = 9$ and $t'_{3} = 3$, then there is no adjacent $4$-faces disjoint from the adjacent $3$-faces, thus $\mu^{*}(f) \geq 3 - 1 \times t_{3}' = 0$.

From now on, we assume that $f$ is an $8$-face. First, assume that the boundary of $f$ is not a cycle. Since $\mathrm{dist}^{\Delta} \geq 2$ and there is no $5$-cycles, the boundary of $f$ consists of two $4$-cycles. In this case, $f$ is not adjacent to any $3$-face and is adjacent to at most two $4$-faces, see \cref{4-8-face}. By \ref{R2}, $f$ sends at most $\frac{1}{4}\times 4$ to the adjacent $4$-faces, thus $\mu^{*}(f) \geq 8 - 6 -\frac{1}{4} \times 4 > 0$. Next, suppose $f$ is bounded by an $8$-cycle. Now, we have $t_{3}' \leq \lfloor \frac{8}{3}\rfloor = 2$. 

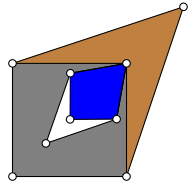
\begin{figure}
\centering
\begin{tikzpicture}
\def\s{1.5}
\coordinate (O) at (0, 0);
\coordinate (NE) at (0.5*\s, 0.5*\s);
\coordinate (W) at (-\s, 0);
\coordinate (S) at (0, -\s);
\coordinate (SW) at (-\s, -\s);
\coordinate (y) at (190:0.5*\s);
\coordinate (m) at (225:1*\s);
\coordinate (q) at (225:0.7*\s);
\coordinate (z) at (260:0.5*\s);
\fill[gray] (O)--(W)--(SW)--(S)--(O)--(z)--(m)--(y)--(O)--(z)--(q)--(y)--cycle;
\fill[brown] (S)--(NE)--(W)--(O);
\fill[blue] (O)--(y)--(q)--(z)--(O);
\draw (S)--(NE)--(W);
\draw (O)--(y)--(m)--(z)--(O)--cycle;
\draw (O)--(y)--(q)--(z)--(O)--cycle;
\draw (W)--(SW)--(S)--(O)--cycle;
\hollownode{O};
\hollownode{NE};
\hollownode{W};
\hollownode{S};
\hollownode{SW};
\hollownode{y};
\hollownode{m};
\hollownode{q};
\hollownode{z};
\end{tikzpicture}
\caption{Two $4$-faces share edges with a common $8$-face.}
\label{4-8-face}
\end{figure}

Suppose $f$ contains at most two maximal $4^{-}$-controlling paths. If $f$ sends at most $1$ via each maximal path, then $\mu^{*}(f) \geq 8 - 6 - 1 \times 2 = 0$. Assume $f$ sends more than $1$ via a maximal path $P$. Then $P$ has length at least four and controls two $3$-faces. It is easy to check that $f$ sends at most $\frac{2}{3} \times 2 = \frac{4}{3}$ via the path $P$, and at most $\frac{1}{2}$ via the other maximal path (if it exists). Hence, $\mu^{*}(f) \geq 8 - 6 -\frac{4}{3} - \frac{1}{2} > 0$. 

Assume $f$ contains at least three maximal $4^{-}$-controlling paths. By parity, we have $s_{0}(f) \leq 2$ and each maximal $4^{-}$-controlling path of $f$ has length at most three. By \cref{Claim1,Claim2,Claim3,Claim4}, it sends at most $\frac{1}{2}$ via each maximal $4$-controlling path, and at most $1$ to each of the other maximal $4^{-}$-controlling path. It follows that $f$ sends at most $1 \times 2 + \frac{1}{2} \times 2 = 3$ in total. 

If $f$ contains a $6^{+}$-vertex, then $\mu^{*}(f) \geq 8 - 6 - 3 + 1 = 0$ by \ref{R5}.

Suppose that $f$ is incident to a semi-rich $5$-vertex $u$. Then it receives $\frac{3}{4}$ from each such vertex by \ref{R4}. Meanwhile, the $5$-vertex $u$ ``saves'' $\frac{1}{4}$ for the expenditure by \cref{Remark}, thus $\mu^{*}(f) \geq 8 - 6 - 3 + \frac{1}{4} + \frac{3}{4} = 0$.

Suppose that $f$ is incident to a poor or a rich $5$-vertex, namely $w$. By parity, $f$ must have precisely three maximal $4^{-}$-controlling paths. Thus, $f$ receives at least $\frac{1}{2}$ from $w$ by \ref{R4}, resulting in $\mu^{*}(f) \geq 8 - 6 - 1\times 2 - \frac{1}{2} + \frac{1}{2} = 0$. 

Assume every vertex on $f$ is a $4^{-}$-vertex. 

\begin{case}
There exists a rich $4$-vertex on $f$.
\end{case}

By parity, $f$ contains precisely three maximal $4^{-}$-controlling paths. Similar to the previous arguments, $f$ sends at most $1 \times 2 + \frac{1}{2} = \frac{5}{2}$. If $f$ can receive at least $\frac{1}{2}$ in total from its incident vertices, then $\mu^{*}(f) \geq 8 - 6 - \frac{5}{2} + \frac{1}{2} = 0$. So we may assume that $f$ can only receive less than $\frac{1}{2}$ in total from its incident vertices. Let $f = [vu_{1}u_{2}xyzw_{2}w_{1}]$ be an $8$-face, and let $v$ be a rich $4$-vertex with four incident faces $f_{1}, f_{2}, f_{3}$ and $f_{4}$ in a cyclic order and $f = f_{3}$. By \cref{lem:N-special8faces}, at most one of $f_{2}$ and $f_{4}$ is a special $8$-face. According to the discharging rules, we have $\tau(v \rightarrow f) = \frac{1}{4}$.

According to the discussions above, we can assume that each rich $4$-vertex on $f$ sends $\frac{1}{4}$ to $f$. It follows that $f$ contains precisely one rich $4$-vertex, namely $v$, for otherwise $\mu^{*}(f) \geq 8 - 6 - \frac{5}{2} + \frac{1}{4} \times 2 = 0$. If $f$ sends at most $\frac{9}{4}$ in total, then $\mu^{*}(f) \geq 8 - 6 - \frac{9}{4} + \frac{1}{4} = 0$. So we may assume that $f$ sends more than $\frac{9}{4}$ in total. Since $f$ contains precisely three maximal $4^{-}$-controlling paths and $f$ is adjacent to at most two $3$-faces, we conclude that there is a maximal $4$-controlling path $P$. Then $f$ sends more than $\frac{1}{4}$ via $P$, so $P$ controls a $(3, 3, 3, 3)$-face or two $(3, 3, 3, 4)$-faces. According to the discharging rules, $f$ sends $\frac{1}{2}$ in total via this maximal $4$-controlling path $P$. Let the other two maximal $4^{-}$-controlling paths be $P_{1}$ and $P_{2}$. Since $f$ sends more than $\frac{9}{4}$ in total, we have that $f$ sends at least $\frac{3}{4}$ via each of $P_{1}$ and $P_{2}$. It follows that neither $P_{1}$ nor $P_{2}$ is a maximal $4$-controlling path by \cref{Claim1,Claim2,Claim3,Claim4}. Note that there are at most seven vertices on maximal $4^{-}$-controlling paths, and one from $\{P_{1}, P_{2}\}$, say $P_{1}$, has length one. Thus, $P_{1}$ controls a $(3, 3, 3)$-face. Moreover, $P_{2}$ controls a $(3, 3, 3)$-face, or $P_{2}$ is a $(3, 4, 3)$-path controlling a $(3, 3, 4)$-face and a $(3, 3, 3, 4)$-face. Since $\tau(v \rightarrow f) = \frac{1}{4}$, we have $f_{2}$ or $f_{4}$, say $f_{4}$, is a special $8$-face. Since $f$ is adjacent to two $3$-faces and $\mathrm{dist}^{\Delta} \geq 2$, the vertex $v$ is not incident with a $3$-face. Thus, $f_{4}$ is a special $8$-face of type I. Now, $G$ contains \cref{C-f3f4a} or \cref{C-f3f4b} as a subgraph.

\begin{case}
There exists no rich $4$-vertices on $f$. 
\end{case}

Recall that $f$ is adjacent to at most two $3$-faces, at most two of the maximal $4^{-}$-controlling paths can control some $3$-faces. Similar to the previous cases, $f$ sends at most $1 \times 2 + \frac{1}{2} \times 2 = 3$. Assume that $f$ contains at least two semi-rich $4$-vertices. By \ref{R3}, $f$ receives at least $\frac{1}{4}$ from each incident semi-rich $4$-vertex. Moreover, each semi-rich $4$-vertex ``saves'' at least $\frac{1}{4}$ by \cref{Remark}. Thus, $\mu^{*}(f) \geq 8 - 6 - 3 + \frac{1}{4} \times 2 + \frac{1}{4} \times 2 = 0$. Assume that there is at most one semi-rich $4$-vertex with respect to $f$.

\begin{subcase}
$f$ contains precisely three maximal $4^{-}$-controlling paths. 
\end{subcase}
Similar to the above cases, $f$ sends at most $1 \times 2 + \frac{1}{2} = \frac{5}{2}$. If $f$ contains a semi-rich $4$-vertex, then it sends $\frac{1}{4}$ to $f$, and it also ``saves'' at least $\frac{1}{4}$ for $f$, yielding $\mu^{*}(f) \geq 8 - 6 - \frac{5}{2} + \frac{1}{4} \times 2 = 0$. So every $4$-vertex on $f$ is poor. Since $f$ contains precisely three maximal $4^{-}$-controlling paths, it contains at most two poor $4$-vertices. If $f$ controls a $(4, 4, 3^{+}, 4^{+})$-face via a $(4, 4)$-edge, then it receives $\frac{1}{2}$ from this $4$-face by \ref{R2}, implying $\mu^{*}(f) \geq 8 - 6 - \frac{5}{2} + \frac{1}{2} = 0$. If $f$ controls a $(4, 4, 4^{+})$-face, then $f$ contains a maximal controlling path of length three and it sends at most $\frac{1}{2}$ via this path, thus $\mu^{*}(f) \geq 8 - 6 - 1 - \frac{1}{2} \times 2 = 0$. Hence, every $(4, 4)$-edge on $f$ controls a $(4, 4, 3)$-face or $(4, 4, 3, 3)$-face.

A controlled $4^{-}$-face $g$ is \emph{good} if it is controlled by a $(3, 3)$-edge with a $4^{+}$-vertex on $g$, or controlled by a $(3, 4)$-edge with at least two $4^{+}$-vertices on $g$. By \cref{Remark}, every good face ``saves'' at least $\frac{1}{4}$ for $f$. If $f$ controls at least two good faces, then it ``saves'' at least $\frac{1}{4} \times 2 = \frac{1}{2}$ in total, and then $\mu^{*}(f) \geq 8 - 6 - \frac{5}{2} + \frac{1}{2} = 0$. So we can assume that $f$ controls at most one good face. 

Assume that $f$ contains no poor $4$-vertex. Then $f$ is adjacent to at least two light $4^{-}$-faces. If all the adjacent light faces are $4$-faces, then $\mu^{*}(f) \geq 8 - 6 - 1 - \frac{1}{2} \times 2 = 0$. Hence, $f$ is adjacent to a light $3$-face, and \cref{Ca} is appeared in $G$, a contradiction.

Assume that $f$ contains precisely one poor $4$-vertex. Then $G$ contains a subgraph as depicted in \cref{Cb} or \cref{Cc}, a contradiction.

Assume that $f$ contains two non-adjacent poor $4$-vertices. Then either \cref{Cd} or \cref{Ce} is a subgraph of $G$, a contradiction.

Assume that $f$ contains two adjacent poor $4$-vertices. Recall that every $(4, 4)$-edge on $f$ controls a $(4, 4, 3)$-face or $(4, 4, 3, 3)$-face. Therefore, $G$ contains a subgraph as depicted in \cref{Cf}, a contradiction. 

\begin{subcase}
$f$ contains precisely four maximal $4^{-}$-controlling paths. 
\end{subcase}
Then each maximal $4^{-}$-controlling path has length one. If each controlled face contains a $4^{+}$-vertex, then each such $4^{+}$-vertex ``saves'' at least $\frac{1}{4}$ for $f$, implying $\mu^{*}(f) \geq 8 - 6 - 3 + \frac{1}{4} \times 4 = 0$. Hence, $f$ controls a light $4^{-}$-face. If $f$ contains no semi-rich $4$-vertex, then each vertex on $f$ is a $3$-vertex, and there is a subgraph as depicted in \cref{Ca}, which leads to a contradiction. 

Next, we consider the case that $f$ contains precisely one semi-rich $4$-vertex, say $x$. If $f$ is adjacent to precisely one light $4^{-}$-face, then $x$ sends $\frac{1}{4}$ to $f$ by \ref{R3}, and each of the three non-light $4^{-}$-face ``saves'' at least $\frac{1}{4}$ for $f$, yielding $\mu^{*}(f) \geq 8 - 6 - 3 + \frac{1}{4} + \frac{1}{4} \times 3 = 0$. Hence, $f$ is adjacent to at least two light $4^{-}$-faces. Suppose that $xy$ controls a $4^{-}$-face $h = [xyz]$ or $h = [xyzu]$. If $h$ contains precisely one $4^{+}$-vertex, then there is a subgraph as depicted in \cref{Cc} or \cref{Cg}, a contradiction. So $h$ contains two $4^{+}$-vertices. 

Assume $h$ is a $4$-face. If $f$ sends at most $\frac{9}{4}$ in total to the other three controlled faces, then $\mu^{*}(f) \geq 8 - 6 - \frac{9}{4} + \frac{1}{4} = 0$, where the $\frac{1}{4}$ is given from $x$ by \ref{R3}. The other case is when $f$ sends more than $\frac{9}{4}$ in total to the other three controlled faces. Then $f$ controls two light $3$-faces and one light $4$-face. Now, $f$ is a special $8$-face, and $x$ sends $\frac{1}{2}$ to $f$ by \ref{R3}, thus $\mu^{*}(f) \geq 8 - 6 - 1 \times 2 - \frac{1}{2} + \frac{1}{2} = 0$. 

Assume $h$ is a $3$-face. Indeed, $h$ is a $(3, 4, 4^{+})$-face, and $f$ sends $\frac{1}{3}$ to $h$. If $f$ sends at most $\frac{23}{12}$ in total to the other three controlled faces, then $\mu^{*}(f) \geq 8 - 6 - \frac{23}{12} - \frac{1}{3} + \frac{1}{4} = 0$. The other case is that $f$ sends more than $\frac{23}{12}$ in total to the other three controlled faces. Hence, $f$ controls one light $3$-face and two light $4$-faces. Now, $f$ must be a special $8$-face. Then $x$ sends $\frac{1}{2}$ to $f$ by \ref{R3}, and $\mu^{*}(f) \geq 8 - 6 - 1 - \frac{1}{2} \times 2 - \frac{1}{3} + \frac{1}{2} > 0$. 
\end{proof}

\begin{figure}
\def\s{1}
\def\8cycle{
\foreach \i in {1, 2, ..., 8}
{
\def\pointname{v\i}
\coordinate (\pointname) at ($(\i*45+22.5:\s)$);
}
\draw (v1)--(v2)--(v3)--(v4)--(v5)--(v6)--(v7)--(v8)--cycle;
\foreach \i in {1, 2, ..., 8}{\solidnode{v\i};}
}
\subcaptionbox{\label{4a}}{\begin{tikzpicture}
\8cycle
\draw (v5) arc (180:360:0.383);
\draw (v1)node[above]{$x_{5}$}--(v2)node[above]{$x_{6}$}--(v3)node[left]{$x_{7}$}--(v4)node[left]{$x_{8}$};
\draw (v5)node[left]{$x_{1}$}--(v6)node[right]{$x_{2}$}--(v7)node[right]{$x_{3}$}--(v8)node[right]{$x_{4}$};
\end{tikzpicture}}
\subcaptionbox{\label{4b}}{\begin{tikzpicture}
\8cycle
\draw (v5) arc (180:360:0.383);
\draw (v8) arc (-45:135:0.383);
\draw (v8) -- ($1.3*(v8)$);
\draw (v1)node[above]{$x_{5}$}--(v2)node[above]{$x_{6}$}--(v3)node[left]{$x_{7}$}--(v4)node[left]{$x_{8}$};
\draw (v5)node[left]{$x_{1}$}--(v6)node[right]{$x_{2}$}--(v7)node[right]{$x_{3}$}--(v8)node[right]{$x_{4}$};
\recsolidnode{v8};
\end{tikzpicture}}
\subcaptionbox{\label{4c}}{\begin{tikzpicture}
\8cycle
\draw (v1) arc (0:180:0.383);
\draw (v5) arc (180:360:0.383);
\draw (v1) -- ($1.3*(v1)$);
\draw (v1)node[right]{$x_{5}$}--(v2)node[left]{$x_{6}$}--(v3)node[left]{$x_{7}$}--(v4)node[left]{$x_{8}$};
\draw (v5)node[left]{$x_{1}$}--(v6)node[right]{$x_{2}$}--(v7)node[right]{$x_{3}$}--(v8)node[right]{$x_{4}$};
\recsolidnode{v1};
\end{tikzpicture}}
\subcaptionbox{\label{4d}}{\begin{tikzpicture}
\8cycle
\draw (v1) arc (0:180:0.383);
\draw (v5) arc (180:360:0.383);
\draw (v6) arc (-135:45:0.383);
\draw (v1) -- ($1.3*(v1)$);
\draw (v1)node[above]{$x_{5}$}--(v2)node[left]{$x_{6}$}--(v3)node[left]{$x_{7}$}--(v4)node[left]{$x_{8}$};
\draw (v5)node[left]{$x_{1}$}--(v6)node[below]{$x_{2}$}--(v7)node[right]{$x_{3}$}--(v8)node[right]{$x_{4}$};
\recsolidnode{v1};
\recsolidnode{v6};
\end{tikzpicture}}
\subcaptionbox{\label{4e}}{\begin{tikzpicture}
\8cycle
\draw (v8) -- ($1.3*(v8)$);
\draw (v3) arc (90:270:0.383);
\draw (v5) arc (180:360:0.383);
\draw (v3) -- ($1.3*(v3)$);
\draw (v8) arc (-45:135:0.383);
\draw (v1)node[above]{$x_{5}$}--(v2)node[above]{$x_{6}$}--(v3)node[left]{$x_{7}$}--(v4)node[left]{$x_{8}$};
\draw (v5)node[left]{$x_{1}$}--(v6)node[right]{$x_{2}$}--(v7)node[right]{$x_{3}$}--(v8)node[right]{$x_{4}$};
\recsolidnode{v8};
\recsolidnode{v3};
\end{tikzpicture}}
\subcaptionbox{\label{4f}}{\begin{tikzpicture}
\8cycle
\draw (v1) arc (0:180:0.383);
\draw (v8) -- ($1.3*(v8)$);
\draw (v8) arc (-45:135:0.383);
\draw (v5) arc (180:360:0.383);
\draw (v1)node[above]{$x_{5}$}--(v2)node[left]{$x_{6}$}--(v3)node[left]{$x_{7}$}--(v4)node[left]{$x_{8}$};
\draw (v5)node[left]{$x_{1}$}--(v6)node[below]{$x_{2}$}--(v7)node[right]{$x_{3}$}--(v8)node[right]{$x_{4}$};
\recsolidnode{v8};
\recsolidnode{v1};
\end{tikzpicture}}
\subcaptionbox{\label{4g}}{\begin{tikzpicture}
\8cycle
\draw (v5) arc (180:360:0.383);
\draw (v7) arc (-90:90:0.383);
\draw (v7) -- ($1.3*(v7)$);
\draw (v1)node[above]{$x_{5}$}--(v2)node[above]{$x_{6}$}--(v3)node[left]{$x_{7}$}--(v4)node[left]{$x_{8}$};
\draw (v5)node[left]{$x_{1}$}--(v6)node[right]{$x_{2}$}--(v7)node[below]{$x_{3}$}--(v8)node[right]{$x_{4}$};
\recsolidnode{v7};
\end{tikzpicture}}
\subcaptionbox{\label{4h}}{\begin{tikzpicture}
\begin{scope}[shift={(-1,0)},rotate=22.5]
\8cycle
\coordinate (H) at ($(v2)!1!60:(v1)$);
\draw (v1) -- (H) -- (v2);
\draw (H)node[above]{$w_{1}$};
\solidnode{H};
\draw (v1)node[above]{$v_{6}$}--(v2)node[left]{$v_{5}$}--(v3)node[left]{$v_{4}$}--(v4)node[below]{$v_{3}$};
\draw (v5)node[below]{$v_{2}$}--(v6)node[below]{$v_{1}$};
\draw (v8)node[above]{$v_{7}$};
\end{scope}
\begin{scope}[shift={(1,0)},rotate=22.5]
\8cycle
\coordinate (H) at ($(v1)!1!60:(v8)$);
\draw (v8) -- (H) -- (v1);
\draw (H)node[above]{$w_{2}$};

\draw (v1)node[above]{$u_{6}$}--(v2)node[above]{$u_{7}$}--(v3)node[right]{$v$}--(v4)node[below]{$u_{1}$};
\draw (v5)node[below]{$u_{2}$}--(v6)node[below]{$u_{3}$}--(v7)node[right]{$u_{4}$}--(v8)node[right]{$u_{5}$};
\solidnode{H};
\end{scope}
\recsolidnode{v3};
\end{tikzpicture}}
\subcaptionbox{\label{4i}}{\begin{tikzpicture}
\begin{scope}[shift={(-0.925,0)}]
\foreach \i in {1, 2, ..., 8}
{
\def\pointname{u\i}
\coordinate (\pointname) at ($(\i*45+22.5:\s)$);
}
\draw (u1)--(u2)--(u3)--(u4)--(u5)--(u6)--(u7)--(u8)--cycle;
\foreach \i in {1, 2, ..., 8}{\solidnode{u\i};}
\coordinate (H) at ($(u2)!1!60:(u1)$);
\draw (u1)node[above]{$u_{6}$}--(u2)node[left]{$u_{5}$}--(u3)node[left]{$u_{4}$}--(u4)node[below]{$u_{3}$};
\draw (u5)node[below]{$u_{2}$}--(u6)node[below]{$u_{1}$}--(u7)node[below]{$u$};
\draw (u8)node[above]{$v$};
\end{scope}
\begin{scope}[shift={(0.925,0)}]
\8cycle
\coordinate (H) at ($(v1)!1!60:(v8)$);
\draw (v8) -- (H) -- (v1);
\solidnode{H};
\coordinate (H1) at ($(v7)!1!90:(v6)$);
\coordinate (H2) at ($(v6)!1!-90:(v7)$);
\draw (v6) -- (H2) -- (H1) -- (v7);
\solidnode{H1};
\solidnode{H2};
\end{scope}
\draw (u6) -- (v5);
\draw (v3) -- ($(u1)-(u8)+(v2)-(v3)$);
\recsolidnode{v3};
\draw (v1)node[above]{$v_{2}$}--(v2)node[left]{$v_{1}$};
\draw (v5)node[below]{$v_{6}$}--(v6)node[below]{$v_{5}$}--(v7)node[right]{$v_{4}$}--(v8)node[right]{$v_{3}$};
\draw (H1)node[right]{$w_{2}$};
\draw (H2)node[below]{$w_{3}$};
\draw (H)node[above]{$w_{1}$};
\end{tikzpicture}}
\subcaptionbox{\label{4j}}{\begin{tikzpicture}
\begin{scope}[shift={(-0.925,0)}]
\foreach \i in {1, 2, ..., 8}
{
\def\pointname{u\i}
\coordinate (\pointname) at ($(\i*45+22.5:\s)$);
}
\draw (u1)--(u2)--(u3)--(u4)--(u5)--(u6)--(u7)--(u8)--cycle;
\foreach \i in {1, 2, ..., 8}{\solidnode{u\i};}
\coordinate (H) at ($(u2)!1!60:(u1)$);
\draw (u1)node[above]{$u_{6}$}--(u2)node[above]{$u_{5}$}--(u3)node[left]{$u_{4}$}--(u4)node[left]{$u_{3}$};
\draw (u5)node[below]{$u_{2}$}--(u6)node[below]{$u_{1}$}--(u7)node[below]{$u$};
\draw (u8)node[above]{$v$};
\end{scope}
\begin{scope}[shift={(0.925,0)}]
\8cycle
\coordinate (H) at ($(v1)!1!60:(v8)$);
\draw (v8) -- (H) -- (v1);
\solidnode{H};
\coordinate (H1) at ($(v7)!1!90:(v6)$);
\coordinate (H2) at ($(v6)!1!-90:(v7)$);
\draw (v6) -- (H2) -- (H1) -- (v7);
\solidnode{H1};
\solidnode{H2};
\draw (v1)node[above]{$v_{2}$}--(v2)node[above]{$v_{1}$};
\draw (v5)node[below]{$v_{6}$}--(v6)node[below]{$v_{5}$}--(v7)node[right]{$v_{4}$}--(v8)node[right]{$v_{3}$};
\draw (H1)node[right]{$w_{2}$};
\draw (H2)node[below]{$w_{3}$};
\draw (H)node[above]{$w_{1}$};
\end{scope}
\draw (u6) -- (v5);
\draw (v3) -- ($(u1)-(u8)+(v2)-(v3)$);
\draw (u4) arc (135:315:0.383);
\draw (u4) -- ($1.3*(u4)$);
\recsolidnode{v3};
\recsolidnode{u4};
\end{tikzpicture}}

\caption{A set of unavoidable configurations, where a semicircle represents a path of length two or three with interior vertices having degree three in $G$, a solid point represents a vertex of degree three, and a green square point represents a vertex of degree four in $G$.}
\label{UR}
\end{figure}
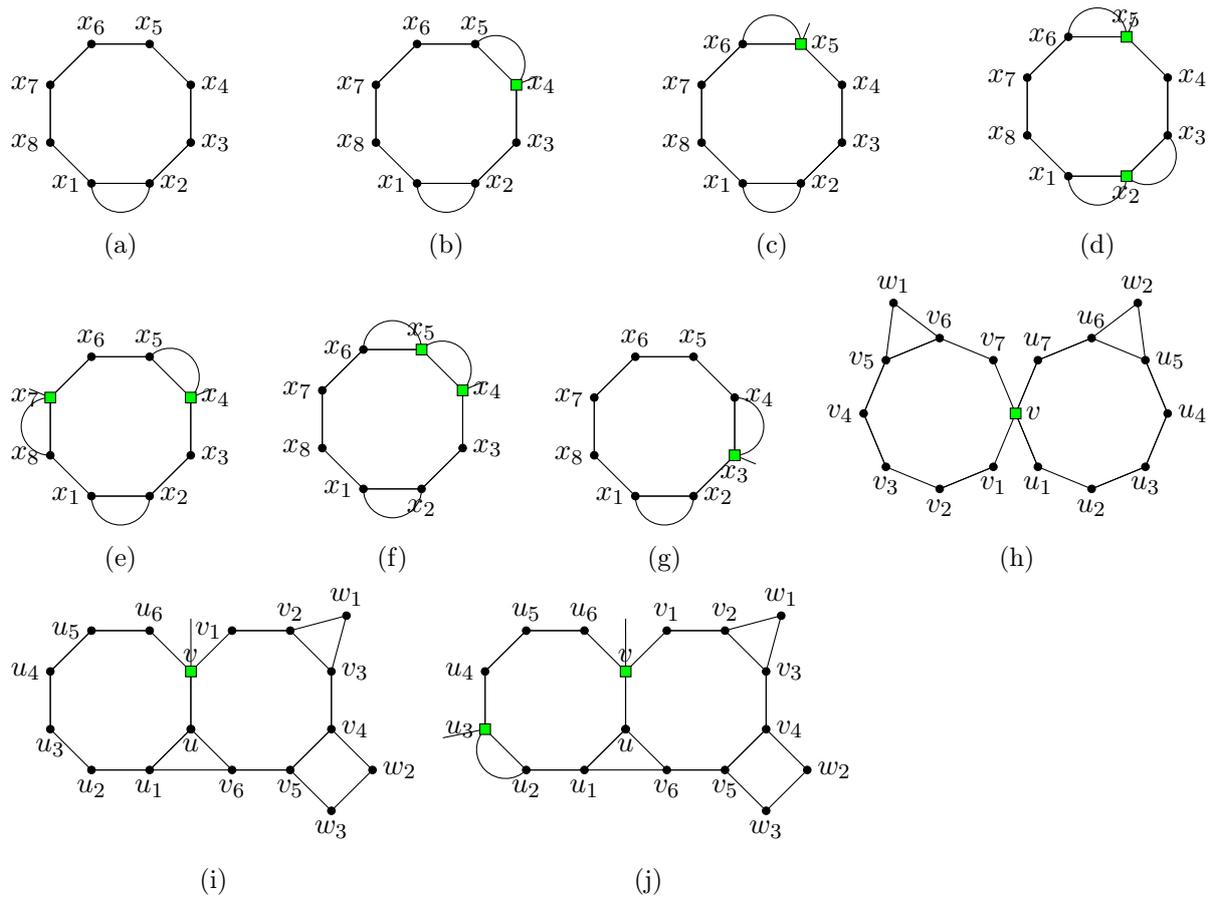

\section{Proof of \cref{Weak}}
\label{sec:3}
In this section, we establish the validity of \cref{Weak}. A graph is a \emph{GDP-tree} if it is connected and each block is either a complete graph or a cycle. Bernshteyn and Lee \cite{MR4606413} demonstrated the following conclusion.

\begin{theorem}\label{wd:minimal}
Assume $G$ is a graph satisfying $\textsf{wd}(G) = d \geq 3$ and $\textsf{wd}(G^{*}) < d$ for each proper subgraph $G^{*}$ of $G$. Then the following statements are true:
\begin{enumerate}[label=(\roman*)]
\item Every vertex has degree at least $d$.
\item For any subset $X \subseteq \{x \in V(G): d_{G}(x)=d\}$, each connected component of $G[X]$ forms a GDP-tree.
\end{enumerate}
\end{theorem}

Now, we can easily prove \cref{Weak}.
\begin{proof}[Proof of \cref{Weak}]
Suppose that $G$ is a counterexample to \cref{Weak} with minimum number of $|V(G)| + |E(G)|$. Thus, $G$ is connected. Moreover, $\textsf{wd}(G) = 3$ and $\textsf{wd}(G^{*}) < 3$ for every proper subgraph $G^{*}$ of $G$. By \cref{STR}, $G$ has a vertex of degree at most two, or it contains a subgraph isomorphic to \cref{UC}. According to \cref{wd:minimal}, a $2^{-}$-vertex is reducible, hence there exists a subgraph isomorphic to \cref{UC}. Let $f(v) = 2$ for each vertex $v \in V(G)$. 

We use $G_{1}$ to denote the graph induced by all vertices in \cref{Ca}. As $G_{1}$ contains the depicted graph as a spanning subgraph, $G_{1}$ has exactly one block that is neither a cycle nor a complete graph. From \cref{wd:minimal}, \cref{Ca} cannot be a subgraph of $G$, which leads to a contradiction.

Let $G_{t}$, $t \in \{2,\dots, 7\}$, represent the subgraph induced by vertices depicted in \cref{Cb,Cc,Cd,Ce,Cf,Cg}. The vertices of $G_{t}$ are labeled as illustrated in \cref{4b,4c,4d,4e,4f,4g}. For each $i \in \{1, \dots, 7\}$, if $x_{i}x_{i+1}$ controls a $4^{-}$-face, then $x_{i}x_{i+1}$ controls $x_{i}u_{i}v_{i}x_{i+1}$ or $x_{i}v_{i}x_{i+1}$. By the minimality, we remove all vertices from $G-V(G_{t})$ by a series of legal operations. Given an ordering: \[x_{2}, x_{3}, \dots, x_{8}, x_{1}, (u_{1}), v_{1},\] if $u_{i}, v_{i}$ ($2 \leq i \leq 7$) exist in $G_{t}$, we insert the vertices $u_{i}, v_{i}$ into the list order following $x_{i}$. Let $g_{t}$ be the resulting function on $G_{t}$. Since $g_{t}(x_{2}) > g_{t}(v_{1})$, we legally remove the vertices from $G_{t}$ by a series of \textsf{Delete} operations except the operation $\textsf{DeleteSave}_{(x_{2},v_{1})}$. Hence, $G$ is weakly $2$-degenerate, which contradicts our assumption. 

Let $G_{8}$ denote the graph induced by vertices depicted in \cref{C-special}. The vertices in $G_{8}$ are labeled as illustrated in \cref{4h}. We order the vertices of $G_{8}$ as: \[v_{6}, v_{7}, u_{6}, u_{7}, v, u_{1}, u_{2}, \dots, u_{5}, w_{2}, v_{1}, v_{2}, \dots, v_{5}, w_{1}.\] By minimality, starting from $(G, f)$, we remove all vertices from $G - V(G_{8})$ by a series of legal operations. Let $f'$ be the resulting function on $G_{8}$. Note that $f'(v_{6}) > f'(w_{1})$ and $f'(u_{6}) > f'(w_{2})$. Beginning from $(G_{8}, f')$, following the given order, we can delete all vertices from $G_{8}$, except for the initial operation $\textsf{DeleteSave}_{(v_{6},w_{1})}$ and the subsequent operation $\textsf{DeleteSave}_{(u_{6},w_{2})}$. Consequently, $G$ is weakly $2$-degenerate, which contradicts our assumption.

Let $G_{9}$ (resp. $G_{10}$) denote the graph induced by the vertices depicted in \cref{C-f3f4a} (resp. \cref{C-f3f4b}). The vertices of $G_{9}$ and $G_{10}$ are as depicted in \cref{4i,4j}. The edge $u_{2}u_{3}$ in \cref{4j} controls a face $u_{2}w_{4}w_{5}u_{3}$ or $u_{2}w_{5}u_{3}$. We order all vertices $S$ of $G_{9}$ and $G_{10}$ as: \[v_{2}, v_{1}, v, u_{6}, u_{5}, u_{4}, u_{3}, (w_{5}), (w_{4}), u_{2}, u_{1}, u, v_{6}, v_{5}, w_{3}, w_{2}, v_{4}, v_{3}, w_{1}.\] Similarly, we remove all vertices from $G-S$ by a sequence of legal operations. Let $f'$ be the resulting function on $S$. Notably, $f'(v_{2}) > f'(w_{1})$, thus we can further legally delete all vertices from $S$ except for the operation $\textsf{DeleteSave}_{(v_{2},w_{1})}$. Consequently, $G$ is weakly $2$-degenerate, which contradicts our assumption.
\end{proof}

\section{Proof of \cref{IF:partition}}
\label{sec:4}
In this section, we establish \cref{IF:partition} utilizing a concept termed strictly $f$-degenerate transversals.

Let $G$ be a simple graph, and for each $v \in V(G)$, let $L_{v} = \{(v, 1), (v, 2), \dots, (v, s)\}$. For each edge $ab \in E(G)$, let $\mathscr{M}_{ab}$ be a matching between $L_{a}$ and $L_{b}$. A \emph{cover} of $G$ is represented by a graph $H$ with the following conditions:
\begin{enumerate}[label = (\roman*)]
\item the vertex set of $H$ is the union of $L_{v}$ for all $v \in V(G)$; and
\item the edge set $\mathscr{M}$ of $H$ is the union of all $\mathscr{M}_{ab}$ whenever $ab$ is an edge of $G$.
\end{enumerate}

Let $H$ be a cover of $G$, and let $f \in \mathbb{Z}^{H}$, forming a \emph{valued cover} $(H, f)$ of $G$. For any subset $S$ of $V(G)$, the induced subgraph $H[\bigcup_{v \in S} L_{v}]$ is simply written as $H_{S}$. A subset $T \subseteq V(H)$ is a \emph{transversal} of $H$ if each $L_{v}$ contains exactly one vertex from $T$. Furthermore, if $H[T]$ is strictly $f$-degenerate, then $T$ is called a \emph{strictly $f$-degenerate transversal} (SfDT for short).

We establish a more robust result as follows:
\begin{theorem}\label{ww}
Let $G$ be a plane graph with $\mathrm{dist}^{\Delta} \geq 2$ and without $5$-, $6$-, $7$-cycles. Assume $H$ is a cover and $f \in \{0, 1, 2\}^{H}$. If each vertex $v \in V(G)$ satisfies $f(v, 1) + \dots + f(v, s) \geq 3$, then $H$ contains a SfDT.
\end{theorem}

To establish \cref{ww}, we require some preliminary results. For a graph $G$ and a valued cover $(H, f)$, it is \emph{minimally non-strictly $f$-degenerate} if $H$ has no SfDT, yet $H - L_{v}$ admits a SfDT for any $v \in V(G)$. 

Define $\mathscr{D}$ as the set of vertices $v \in V(G)$ satisfying $f(v, 1) + \dots + f(v, s) \geq d_{G}(v)$. Lu \etal \cite{MR4357325} presented the next result on critical graphs. 
\begin{theorem}[Lu \etal \cite{MR4357325}]\label{SFDT:minimal}
Let $G$ be a graph and $(H, f)$ be a valued cover of $G$. Let $D'$ be a nonempty subset of $\mathscr{D}$ with $G[D']$ having no cut vertices. If $(H, f)$ is a minimally non-strictly $f$-degenerate pair, then
\begin{enumerate}[label = (\roman*)]
\item\label{M1} $G$ is connected and $f(v, 1) + f(v, 2) + \dots + f(v, s) \leq d_{G}(v)$ for each $v \in V(G)$, and 
\item\label{M2} $G[D']$ is a cycle, a complete graph, or $d_{G[D']}(v) \leq \max_{q} \{f(v, q)\}$ for each $v \in D'$. \qed
\end{enumerate}
\end{theorem}

\begin{theorem}[Wang \etal \cite{Wang2019+}]\label{WWY}
Let $k$ be an integer with $k \geq 3$, and let $K$ be an induced subgraph of $G$ such that the vertices of $K$ can be ordered as $u_{1}, u_{2}, \dots, u_{m}$ satisfying:
\begin{enumerate}[label = (\roman*)]
\item\label{WWY-1} $k - (d_{G}(u_{1}) - d_{K}(u_{1})) > k - (d_{G}(u_{m}) - d_{K}(u_{m}))$; 
\item\label{WWY-2} $d_{G}(u_{m}) \leq k$ and $u_{1}u_{m}$ is an edge in $G$; and
\item\label{WWY-3} for $2 \leq i \leq m - 1$, $u_{i}$ has at most $k - 1$ neighbors in $G - \{u_{i+1}, \dots, u_{m}\}$.
\end{enumerate}
Let $H$ be a cover of $G$ and $f \in \{0, 1, 2\}^{H}$. If $f(v, 1) + \dots + f(v, s) \geq k$ for each vertex $v \in V(G)$, then any SfDT of $H - H_{K}$ can be extended to that of $H$. \qed
\end{theorem}

Now, we prove \cref{ww}.

\begin{proof}[Proof of \cref{ww}]
Assume that $G$ is a counterexample to \cref{ww} such that $(H,f)$ contains no SfDT, yet $(H - L_{v}, f)$ has a SfDT for any $v \in V(G)$. According to \cref{STR}, $G$ must contain a $2^{-}$-vertex, or a subgraph isomorphic to \cref{UC}. By \cref{SFDT:minimal}, every vertex has degree at least three, thereby indicating the presence of a subgraph isomorphic to \cref{UC} in $G$.

In the following, we still define $G_{1}, G_{2}, \dots, G_{10}$ to be the graph induced by the vertices in \cref{Ca,Cb,Cc,Cd,Ce,Cf,Cg,C-special,C-f3f4a,C-f3f4b} respectively, as that of in the proof of \cref{Weak}, and the vertices of $G_{t}$ are labeled as in \cref{4a,4b,4c,4d,4e,4f,4g,4h,4i,4j}.

Note that $G_{1}$ is a $2$-connected graph that is neither a cycle nor a complete graph. Furthermore, within this subgraph, there exists a vertex $v$ with degree greater than two. This contradicts \cref{SFDT:minimal}.

Next we consider $G_{t}$ for $t \in \{2, \dots, 7\}$. For each $i \in \{1, \dots, 7\}$, if $x_{i}x_{i+1}$ controls a $4^{-}$-face, then $x_{i}x_{i+1}$ controls $x_{i}u_{i}v_{i}x_{i+1}$ or $x_{i}v_{i}x_{i+1}$. Given an ordering: \[x_{2},x_{3},\dots, x_{8},x_{1},(u_{1}),v_{1},\] if $u_{i}, v_{i}$ ($2 \leq i \leq 7$) exist in $G_{t}$, we insert the vertices of $u_{i}, v_{i}$ to the list order following $x_{i}$. Denoting the resulting list as $S$. It is observed that $S$ satisfies the conditions in \cref{WWY} with $k = 3$ and pair $(x_{2}, v_{1})$. By minimality, $(H-H_{S},f)$ has a SfDT $T$. According to \cref{WWY}, we can obtain a SfDT of $H$ by extending $T$, leading to a contradiction.

When we consider $G_{8}$, we order the vertices of $G_{8}$ as follows: \[v_{6},v_{7},u_{6},u_{7},v,u_{1},u_{2},\dots,u_{5},w_{2},v_{1},v_{2},\dots,v_{5},w_{1}.\] It is evident that this ordering satisfies the conditions in \cref{WWY} with $k = 3$ and the pairs $(v_{6}, w_{1})$ and $(u_{6}, w_{2})$. A similar contradiction is derived as above cases. 
 
Now we consider $G_{9}$ and $G_{10}$. Note that the edge $u_{2}u_{3}$ in \cref{4j} controls a face $u_{2}w_{4}w_{5}u_{3}$ or $u_{2}w_{5}u_{3}$. We order all vertices $S$ of $G_{9}$ and $G_{10}$ as follows: \[v_{2}, v_{1}, v, u_{6}, u_{5}, u_{4}, u_{3}, (w_{5}), (w_{4}), u_{2}, u_{1}, u, v_{6}, v_{5}, w_{3}, w_{2},v_{4},v_{3},w_{1}.\] It is notable that this list complies with the conditions specified in \cref{WWY} with $k = 3$ and the pair $(v_{2}, w_{1})$. A similar contradiction is derived as above cases. 
\end{proof}

For a cover $H$ of $G$, we define $H$ as a \emph{good cover} of $G$ if it has a property that $(u, i)$ and $(v, j)$ are adjacent iff $i = j$ and $uv \in E(G)$. Observe that a good cover of $G$ is isomorphic to $s$ copies of $G$.

\begin{proof}[Proof of \cref{IF:partition}]
Let $s = 2$, and consider a good cover $H$ of $G$. Let $f$ be a mapping with $f(v, 1) = 1$ and $f(v, 2) = 2$ for every vertex $v \in V(G)$. According to \cref{ww}, $H$ has a SfDT $T$. Let $\mathcal{I}=\{v: (v, 1) \in T\}$ and $\mathcal{F}=\{v: (v, 2) \in T\}$. It is observed that $\mathcal{I}$ is an independent set in $G$, and $\mathcal{F}$ induces a forest in $G$. Consequently, $G$ is \emph{$(\mathcal{I}, \mathcal{F})$}-partitionable. This establishes \cref{IF:partition}.
\end{proof}

\vskip 0mm \vspace{0.3cm} \noindent{\bf Acknowledgments.} We thank the two anonymous referees for their valuable comments and constructive suggestions on the manuscript. The first author was supported by the Natural Science Foundation of Henan Province (No. 242300420238). The third author was supported by National Natural Science Foundation of China (No. 12101187).


\end{document}